\documentclass[12pt]{article}

\usepackage{mathrsfs}

\usepackage{amsmath,amsthm,amssymb}
\usepackage{mathptmx}
\usepackage{graphicx}

\font\tencmmib=cmmib10 \skewchar\tencmmib '60
\newfam\cmmibfam
\textfont\cmmibfam=\tencmmib

\def\lessim{\ \lower4pt\hbox{$
\buildrel{\displaystyle <}\over\sim$}\ }
\def\gessim{\ \lower4pt\hbox{$\buildrel{\displaystyle >}
\over\sim$}\ }

\newcommand{\la}{\langle}
\newcommand{\ra}{\rangle}

\newcommand{\e}{\mathbb{E}}

\newcommand{\Reals}{\mathbb{R}}

\newtheorem{theorem}{\bf Theorem}

\newtheorem{example}{\bf Example}

\newenvironment{Proof of lemma}{\noindent{\bf Proof of Lemma}}{\hfill$\Box$\newline}
\newenvironment{Proof of theorem}{\noindent{\it Proof of Theorem}}{\hfill\scriptsize{$\Box$}\newline}
\newenvironment{Proof of theorems}{\noindent{\bf Proof of Theorems}}{\hfill$\Box$\newline}
\newenvironment{Proof of proposition}{\noindent{\bf Proof of Proposition}}{\hfill$\Box$\newline}
\newenvironment{Proof of propositions}{\noindent{\bf Proof of Propositions}}{\hfill$\Box$\newline}
\newenvironment{Proof of exercise}{\noindent{\it Proof of Exercise:}}{\hfill$\Box$}

\title{Some examples of quenched self-averaging\\ in models with Gaussian disorder}

\author{Wei-Kuo Chen\thanks{Department of Mathematics, University of Chicago. Email: wkchen@math.uchicago.edu.}
\and
Dmitry Panchenko\thanks{Department of Mathematics, University of Toronto. Email: panchenk@math.toronto.edu.}}

\date{}

\begin{document}

\maketitle

\begin{abstract} 
In this paper we give an elementary approach to several results of Chatterjee in \cite{Chatt09, RFIM}, as well as some generalizations. First, we prove quenched disorder chaos for the bond overlap in the Edwards-Anderson type models with Gaussian disorder. The proof extends to systems at different temperatures and covers a number of other models, such as the mixed $p$-spin model, Sherrington-Kirkpatrick model with multi-dimensional spins and diluted $p$-spin model. Next, we adapt the same idea to prove quenched self-averaging of the bond magnetization for one system and use it to show quenched self-averaging of the site overlap for random field models with positively correlated spins. Finally, we show self-averaging for certain modifications of the random field itself. 
\end{abstract}
\vspace{0.5cm}

\section{Introduction}

The approach developed in this paper was motivated by several results of Chatterjee in \cite{Chatt09, RFIM}. One of the results in \cite{Chatt09} described a quenched disorder chaos for the bond overlap in the setting of the Edwards-Anderson type spin glass models. Consider a finite undirected graph $(V,E)$ and the Edwards-Anderson type Hamiltonian
\begin{equation}
H(\sigma) = \sum_{(i,j)\in E} g_{i,j}\sigma_i\sigma_j,
\label{HamEA}
\end{equation}
where $\sigma = (\sigma_i)_{i\in V} \in \{-1,+1\}^V$ and $g_{i,j}$ are i.i.d. standard Gaussian random variables. Given an inverse temperature parameter $\beta>0$, the corresponding Gibbs measure is defined by 
\begin{equation}
G(\sigma) = \frac{\exp\beta H(\sigma)}{Z},
\end{equation}
where $Z = \sum_{\sigma} \exp \beta H(\sigma)$ is called the partition function. Now, let us consider two copies of this system with different disorder parameters $(g_{i,j}^1)$ and $(g_{i,j}^2)$. We will denote the Hamiltonians and Gibbs measures of these systems by $H_1(\sigma), H_2(\rho)$ and $G_1(\sigma), G_2(\rho)$. Suppose that the disorder parameters of these two systems are correlated,
\begin{equation}
\e g_{i,j}^1 g_{i,j}^2 = t,
\end{equation}
for some $t\in [0,1].$ We still assume that $(g_{i,j}^1, g_{i,j}^2)$ are independent for different $(i,j)\in E.$ When $t=1$, this gives us two copies of the same system, and the interesting case is when $t$ is slightly smaller than one, so the interaction parameters of these two systems are slightly decoupled. Note that in \cite{Chatt09} and \cite{Chatt14} the correlation was written as $e^{-2s}$ for $s\in [0,\infty)$, which is the same as our $t=e^{-2s}$. Consider i.i.d. samples $(\sigma^\ell)_{\ell\geq 1}$ from $G_1$ and $(\rho^\ell)_{\ell\geq 1}$ from $G_2$. The quantity 
\begin{equation}
Q_{\ell,\ell'} = \frac{1}{|E|} \sum_{(i,j)\in E} \sigma_i^\ell \sigma_j^\ell \rho_i^{\ell'} \rho_j^{\ell'}
\end{equation}
is called the bond overlap of configurations $\sigma^\ell$ and $\rho^{\ell'}$, which is a measure of similarity between bonds in these two configurations. Of course, one can similarly define the bond overlap of $\sigma^\ell$ and $\sigma^{\ell'}$, but here one is interested in the behavior of the bond overlap between two slightly decoupled systems. Up to a normalization factor $|E|$, the bond overlap is the covariance
$$
\e H(\sigma^\ell) H(\rho^{\ell'}) = |E| Q_{\ell,\ell'}
$$
of the Hamiltonian $H$ in (\ref{HamEA}). Let us denote by $\la\,\cdot\,\ra$ the average with respect to $(G_1\times G_2)^{\otimes \infty}$. 

\smallskip
In Theorem $1.7$ in \cite{Chatt09} (see Theorem $11.5$ in \cite{Chatt14}), Chatterjee proved that, for any $0<t< 1$,
\begin{align}
\e \Bigl\la \Bigl(Q_{1,1}-\bigl\la Q_{1,1}\bigr\ra \Bigr)^2\Bigr\ra
&\leq \frac{2\sqrt{2}}{\beta t^{1/4}\sqrt{|E| \log (1/t)}}.
\label{eqChatt1}
\end{align} 

\medskip
\noindent
This shows that for $t<1$ and large $|E|$, the bond overlap $Q_{1,1}$ between replicas from these two systems concentrates around its Gibbs average $\la Q_{1,1}\ra$. The first goal of this paper will be to give an elementary proof of essentially the same inequality,
\begin{align}
\e \Bigl\la \Bigl(Q_{1,1}-\bigl\la Q_{1,1}\bigr\ra \Bigr)^2\Bigr\ra
&\leq \frac{8}{\beta\sqrt{|E|(1-t)}},
\label{main1}
\end{align} 
as well as some generalizations. First of all, in addition to the proof being elementary, we get a better dependence on $t$ when $t$ approaches zero, which covers the case $t=0$. In the case when $t$ is close to $1$, the dependence on $t$ is the same, since $\log(1/t)$ is of order $1-t$ in that case. Moreover, the same proof will give us quenched disorder chaos for two systems with different inverse temperature parameters $\beta_1$ and $\beta_2$, in which case (\ref{main1}) will be replaced by
\begin{align}
\e \Bigl\la \Bigl(Q_{1,1}-\bigl\la Q_{1,1}\bigr\ra \Bigr)^2\Bigr\ra
&\leq \frac{4(\beta_1+\beta_2)}{\beta_1\beta_2\sqrt{|E|(1-t)}}.
\label{main2}
\end{align} 
It is not clear to us how to extend Chatterjee's proof to this case, since it seems to rely on the symmetry between two systems in an essential way.  In Section \ref{SecLabDIS}, we will formulate a general disorder chaos result that will cover other examples in addition to the Edwards-Anderson type models, such as the mixed $p$-spin model, Sherrington-Kirkpatrick model with multi-dimensional spins, and diluted $p$-spin model. 

\smallskip
In the second paper, \cite{RFIM}, Chatterjee studied the random field Ising model on the $d$-dimensional lattice with the Hamiltonian
\begin{equation}
H(\sigma) = \beta \sum_{i\sim j}\sigma_i\sigma_j + h\sum_{i} g_i \sigma_i,
\end{equation}
where $\sigma\in\{-1,+1\}^V$ for $V= \mathbb{Z}^d\cap [1,N]^d$, $i\sim j$ means that $i$ and $j$ are neighbors on this lattice, $\beta, h>0,$ and $g_i$ are i.i.d. standard Gaussian random variables. The main goal in \cite{RFIM} was to show that for almost all values $\beta$ and $h$, in the thermodynamic limit, the site overlap
$$
R_{1,2} = \frac{1}{|V|} \sum_{i\in V} \sigma_i^1\sigma_i^2
$$
between two replicas $\sigma^1$ and $\sigma^2$ concentrates around a constant that depends only on $\beta$ and $h$. We are not going to reproduce the entire proof, but will give simplified proofs of two key steps. The first key step was to show quenched self-averaging of the overlap,
\begin{align}
\e \bigl\la \bigl(R_{1,2}-\la R_{1,2}\ra \bigr)^2 \bigr\ra\leq \frac{2\sqrt{2+h^2}}{h\sqrt{|V|}},
\end{align}
as a consequence of positive correlation of spins, which in this model follows from the FKG inequality \cite{FKG}. Our approach in Section \ref{SecLabelFKG} will also remove the factor $\sqrt{2+h^2}$. It will be based on some general result about quenched self-averaging of the bond magnetization for one system in Section \ref{SecLabelMag}. 

\smallskip
Another key step in \cite{RFIM} was to show that the normalized random field
$$
h(\sigma) = \frac{1}{|V|} \sum_{i\in V} g_i \sigma_i
$$
concentrates around its quenched average $\la h(\sigma)\ra$,
\begin{equation}
\e \bigl( \bigl\la h(\sigma)^2\bigr\ra - \bigl\la h(\sigma)\bigr\ra^2 \bigr)
\leq
\frac{\sqrt{24}}{h\sqrt{|V|}} + \frac{1}{|V|}.
\label{EqCh3}
\end{equation}
This step holds more generally and does not depend on the condition that the spins are positively correlated. Again, we will give a simplified proof of a slightly improved bound in Section \ref{SecLabHam} (see equations (\ref{eqlast}) and (\ref{eqlast2})), as well as certain generalizations (the most general statement appears in Theorem \ref{ThLast} in that section). All the proofs will be variations of the same idea and will follow the same simple pattern.

\section{Quenched disorder chaos}\label{SecLabDIS}
We will formulate the main result of this section in a way that will cover a number of models as examples. We will consider two systems with the Hamiltonians
\begin{align}
\label{eq1}
Y_1(\sigma)&=\sum_{e\in E}g_e^1f_e(\sigma),\\
Y_2(\rho)&=\sum_{e\in E}g_e^2f_e(\rho),
\end{align}
defined on the same measurable space $(\Sigma,\mathcal{F})$ (i.e. both $\sigma, \rho\in \Sigma$), which will usually be some finite set. Here the set $E$ is some finite index set, $(f_e)_{e\in E}$ is a family of measurable functions $f_e:\Sigma \rightarrow [-1,1]$, and $(g_e^1,g_e^2)$ are independent Gaussian random pairs for $e\in E$ such that
\begin{equation}
\e (g_e^1)^2=\e (g_e^2)^2=1 \,\,\mbox{ and }\,\, \e g_e^1 g_e^2=t
\label{tcorr}
\end{equation}
for some $t\in[0,1].$ We can allow the functions $(f_e)_{e\in E}$ be random as long as their randomness is independent of the Gaussian random variables $(g_e^1,g_e^2)$, but in all the examples below they will be non-random. In some models, such as diluted models, the cardinality of the index set $E$ can be random as well and, in that case, we will also assume it to be independent of the Gaussian random variables $(g_e^1,g_e^2)$. We will state our result for a fixed $E$, since one can average in $|E|$ later, as we will do, for example, in the diluted models.

Next, we consider the corresponding Gibbs measures $G_1$ and $G_2$ on $(\Sigma,\mathcal{F})$,
\begin{align}
dG_1(\sigma)&=\frac{\exp \gamma_1 Y_1(\sigma)}{Z_1} d\mu_1(\sigma),
\label{Geq1}
\\
dG_2(\rho)&=\frac{\exp\gamma_2 Y_2(\rho)}{Z_2} d\mu_2(\rho),
\end{align}
where $\gamma_1,\gamma_2> 0$ are some fixed constants, $\mu_1$ and $\mu_2$ are random finite measures on $(\Sigma,\mathcal{F})$ and $Z_1,Z_2$ are the partition functions. The randomness of $\mu_1$ and $\mu_2$ should be independent of the Gaussian random variables $(g_e^1, g_e^2)$ but not necessarily of other random variables or each other. As above, we will consider i.i.d. replicas $(\sigma^\ell)_{\ell\geq 1}$ from $G_1$ and $(\rho^\ell)_{\ell\geq 1}$ from $G_2$, let $\la\,\cdot \, \ra$ denote the average with respect to $(G_1\times G_2)^{\otimes \infty}$, and define the overlaps by
\begin{align}
Q_{\ell,\ell'}&=\frac{1}{|E|}\sum_{e\in E}f_e(\sigma^{\ell})f_e(\rho^{\ell'}).
\end{align}
Then the following quenched disorder chaos for the overlap holds.
\begin{theorem}\label{thm1}
If $\gamma_1,\gamma_2>0$ and $t\in[0,1)$ then
\begin{align}
\begin{split}
\label{thm1:eq1}
\e\Bigl \la \Bigl(Q_{1,1}-\bigl\la Q_{1,1}\bigr\ra \Bigr)^2 \Bigr\ra
&\leq \frac{4(\gamma_1+\gamma_2)}{\gamma_1\gamma_2\sqrt{|E|(1-t)}}.
\end{split}
\end{align}
\end{theorem}
\begin{proof} The proof is based on a simple computation first used in the derivation of the (two-system) Ghirlanda-Guerra identities for the mixed $p$-spin model in Chen, Panchenko \cite{CP12} and Chen \cite{C12} (for related results about disorder chaos, see also \cite{Chen11}). Because of the assumption (\ref{tcorr}), we can represent
\begin{align*}
Y_1(\sigma)&=\sqrt{t}Z(\sigma)+\sqrt{1-t}Z_1(\sigma),\\
Y_2(\rho)&=\sqrt{t}Z(\rho)+\sqrt{1-t}Z_2(\rho),
\end{align*}
where, given i.i.d. standard Gaussian random variables $z_e$, $z_e^1$ and $z_{e}^2$ indexed by $e\in E$,
\begin{align*}
Z(\sigma)&=\sum_{e\in E}z_ef_e(\sigma),\,\, Z(\rho)=\sum_{e\in E}z_ef_e(\rho),\\
Z_1(\sigma)&=\sum_{e\in E}z_e^1f_e(\sigma),\,\, Z_2(\rho)=\sum_{e\in E}z_e^2f_e(\rho).
\end{align*}
Let us consider the quantity
$$
\e \Bigl \la Q_{1,1} \frac{Z_1(\rho^1)}{|E|}  \Bigr\ra.
$$ 
Notice that $Z_1(\rho^1)$ is a new object, with the randomness coming from the second term in the Hamiltonian $Y_1$ on the first system, and the argument $\rho^1$ that is a replica from the second system and is averaged with respect to $G_2.$ As a result, if $\e'$ denotes the expectation in the Gaussian random variables $z_e$, $z_e^1$ and $z_{e}^2$, then
$$
\e' Y_1(\sigma^\ell) Z_1(\rho^1) = \sqrt{1-t} |E| Q_{\ell, 1},\,\,
\e' Y_2(\rho^\ell) Z_1(\rho^1) = 0,
$$
and the usual Gaussian integration by parts (see e.g. \cite{SG}, Appendix A.4) gives
\begin{align*}
\e \Bigl \la Q_{1,1} \frac{Z_1(\rho^1)}{|E|}  \Bigr\ra &=\gamma_1\sqrt{1-t}
\e\bigl\la Q_{1,1}^2 - Q_{1,1} Q_{2,1}\bigr\ra.
\end{align*}
On the other hand, since $|Q_{1,1}|\leq 1$, 
$$
\Bigl | \e \Bigl \la Q_{1,1} \frac{Z_1(\rho^1)}{|E|} \Bigr\ra \Bigr | 
\leq 
\e \Bigl \la \frac{|Z_1(\rho^1)|}{|E|} \Bigr\ra.
$$
The average on the right hand side is with respect to $dG_2(\rho^1)$ only, which is independent of the Gaussian random variables $z_e^1$ that appear in $Z_1(\rho)$, so
$$
\e \Bigl \la \frac{|Z_1(\rho^1)|}{|E|} \Bigr\ra
=
\e \Bigl \la \frac{\e_1 |Z_1(\rho^1)|}{|E|} \Bigr\ra,
$$
where $\e_1$ is the expectation with respect to $(z_{e}^1)_{e\in E}$. Finally, since
\begin{align*}
\e_1 |Z_1(\rho^1)|\leq 
\bigl(\e_1 Z_1(\rho^1)^2 \bigr)^{1/2}
=
\bigl( \sum_{e\in E}f_e(\rho^1)^2 \bigr)^{1/2}\leq |E|^{1/2},
\end{align*}
we prove that
\begin{align}\label{thm1:proof:eq2}
\Bigl| \gamma_1\sqrt{1-t}
\e\bigl\la Q_{1,1}^2 - Q_{1,1} Q_{2,1}\bigr\ra\Bigr|
=
\Bigl | \e \Bigl \la Q_{1,1} \frac{Z_1(\rho^1)}{|E|} \Bigr\ra \Bigr | 
\leq \frac{1}{\sqrt{|E|}}.
\end{align}
Next, by symmetry, $\la Q_{2,1}^2 \ra = \la Q_{1,1}^2\ra$ and, therefore, 
\begin{align*}
\e \bigl\la (Q_{1,1}-Q_{2,1})^2 \bigr\ra = 2 \e \bigl\la Q_{1,1}^2 - Q_{1,1}Q_{2,1} \bigr\ra 
\leq \frac{2}{\gamma_1\sqrt{|E|(1-t)}},
\end{align*}
where in the last inequality we used (\ref{thm1:proof:eq2}). Similarly, one can show that
\begin{align*}
\e \bigl\la (Q_{2,2}-Q_{2,1})^2 \bigr\ra  
\leq \frac{2}{\gamma_2\sqrt{|E|(1-t)}}.
\end{align*}
Combining the above two inequalities and using Jensen's inequality,
\begin{align*}
\e \bigl\la \bigl(Q_{1,1}- \la Q_{1,1} \ra \bigr)^2 \bigr\ra
&\leq  \e \bigl\la (Q_{1,1}-Q_{2,2})^2 \bigr\ra \\
& \leq 2 \e \bigl\la (Q_{1,1}-Q_{2,1})^2 \bigr\ra   
 +2\e \bigl\la (Q_{2,2}-Q_{2,1})^2 \bigr\ra \\
&\leq \frac{4}{\gamma_1\sqrt{|E|(1-t)}}+\frac{4}{\gamma_2\sqrt{|E|(1-t)}}.
\end{align*}
This finishes the proof.
\end{proof}

We will now give several examples of applications of Theorem \ref{thm1}. Since all the arguments are very similar, we will only give a detailed discussion of the mixed $p$-spin model.

\begin{example}[mixed $p$-spin model]\rm The Hamiltonian of the mixed $p$-spin model is given by
\begin{align*}
H(\sigma)&=\sum_{p\geq 1}\frac{\beta_p}{N^{(p-1)/2}}\sum_{1\leq i_1,\ldots,i_p\leq N}g_{i_1,\ldots,i_p}\sigma_{i_1}\cdots\sigma_{i_p},
\end{align*}
where $\sigma\in \Sigma_N:=\{-1,+1\}^N$, $(\beta_{p})_{p\geq 1}$ is a sequence of inverse temperature parameters such that $\beta_p\geq 0$ for all $p\geq 1$ and $\sum_{p\geq 1}2^p\beta_p^2<\infty$, and $g_{i_1,\ldots,i_p}$ are i.i.d standard Gaussian for all $p\geq 1$ and all $1\leq i_1,\ldots,i_p\leq N$. Let us now consider two such systems,
\begin{align*}
H_1(\sigma)&=\sum_{p\geq 1}\frac{\beta_{1,p}}{N^{(p-1)/2}}\sum_{1\leq i_1,\ldots,i_p\leq N}g_{i_1,\ldots,i_p}^1\sigma_{i_1}\cdots\sigma_{i_p},\\
H_2(\rho)&=\sum_{p\geq 1}\frac{\beta_{2,p}}{N^{(p-1)/2}}\sum_{1\leq i_1,\ldots,i_p\leq N}g_{i_1,\ldots,i_p}^2\rho_{i_1}\cdots\rho_{i_p},
\end{align*} 
with the Gaussian interaction parameters coupled according to some sequence $(t_p)_{p\geq 1}$,
$$
\e (g_{i_1,\ldots,i_p}^1)^2=\e (g_{i_1,\ldots,i_p}^2)^2=1 
\,\,\mbox{ and }\,\,
\e g_{i_1,\ldots,i_p}^1g_{i_1,\ldots,i_p}^2=t_p \in [0,1].
$$
Suppose that for some $p\geq 1,$ $\beta_{1,p},\beta_{2,p}>0$ and $t_{p}<1$. Let $Y_1$ and $Y_2$ be the $p$-spin terms in $H_1$ and $H_2$ correspondingly. This means that in \eqref{eq1}, we should set $E=\{1,\ldots,N\}^{p}$, for $e=(i_1,\ldots,i_p)\in E$ define $f_e(\sigma)=\sigma_{i_1}\cdots \sigma_{i_p}$ for all $\sigma\in\Sigma_N,$ and let
\begin{align*}
\gamma_1&=\frac{\beta_{1,p}}{N^{(p-1)/2}}
\,\,\mbox{ and }\,\,
\gamma_2=\frac{\beta_{2,p}}{N^{(p-1)/2}}.
\end{align*}
In this case, the bond overlap $Q_{1,1}$ will be equal to
$$
Q_{1,1} = \frac{1}{N^p} \sum_{1\leq i_1,\ldots,i_p\leq N} \sigma_{i_1}^1 \cdots\sigma_{i_p}^2 \rho_{i_1}^1\cdots\rho_{i_p}^1 = (R_{1,1})^p
$$
where $R_{1,1}=N^{-1}\sum_{i=1}^N\sigma_i^1\rho_i^1$ is the usual site overlap. Finally, we can write the Gibbs measures corresponding to $H_1$ and $H_2$ as
$$
G_1(\sigma)=\frac{\exp \gamma_1Y_1(\sigma)}{Z_1}\mu_1(\sigma),\,\,
G_2(\rho)=\frac{\exp \gamma_2Y_2(\rho)}{Z_2}\mu_2(\rho),
$$
where we denoted
$$
\mu_1(\sigma) = \exp \bigl(H_1(\sigma)-\gamma_1 Y_1(\sigma) \bigr),\,\,
\mu_2(\rho) = \exp \bigl(H_2(\rho)-\gamma_2 Y_2(\rho) \bigr).
$$
By construction, these measures are independent of the Gaussian random variables in $Y_1$ and $Y_2$. Theorem \ref{thm1} implies that
\begin{align}
\e\Bigl \la  \Bigl( (R_{1,1})^{p}-\bigl\la (R_{1,1})^{p} \bigr\ra \Bigr)^2  \Bigr\ra
&\leq\frac{4(\gamma_{1}+\gamma_{2})}{\gamma_{1}\gamma_{2}\sqrt{N^p(1-t_{p})}}
=\frac{4(\beta_{1,p}+\beta_{2,p})}{\beta_{1,p}\beta_{2,p}\sqrt{N(1-t_{p})}}.
\end{align}
Clearly, for odd $p$ this implies that $R_{1,1}\approx \la R_{1,1}\ra$ and for even $p$ this implies that $|R_{1,1}| \approx \la |R_{1,1}|\ra$. This example was one of the main results in \cite{CP12}.

\end{example}

\begin{example}[SK model with multidimensional spins]\rm
Let $S$ be a bounded Borel measurable subset of $\mathbb{R}^d$ and $\nu$ be a probability measure on $\mathscr{B}(S)$. Consider the configuration space
$$
\Sigma_N=\left\{(x_1,\ldots,x_N):x_1=(x_{1,u})_{1\leq u\leq d},\ldots,x_N=(x_{N,u})_{1\leq u\leq d}\in S\right\}.
$$
Consider the Hamiltonians and Gibbs measures of two SK type models with multidimensional spins on $\Sigma_N,$
\begin{align*}
H_1(\sigma)&=\frac{\beta_1}{\sqrt{N}}\sum_{1\leq i,j\leq N}g_{i,j}^1(\sigma_i,\sigma_j),\,\,
dG_1(\sigma)=\frac{\exp H_1(\sigma)}{Z_1}d\nu(\sigma),\\
H_2(\rho)&=\frac{\beta_2}{\sqrt{N}}\sum_{i\leq i,j\leq N}g_{i,j}^2(\rho_i,\rho_j),\,\,
dG_2(\rho)=\frac{\exp H_2(\rho)}{Z_2}d\nu(\rho),
\end{align*}
where $(a,b)$ is the scalar product on $\Reals^d$, $\beta_1,\beta_2 > 0$, and $(g_{i,j}^1,g_{i,j}^2)$ are independent Gaussian random vectors with covariance 
$$
\e (g_{i,j}^1)^2=\e(g_{i,j}^2)^2=1 \,\,\mbox{ and }\,\, \e g_{i,j}^1g_{i,j}^2=t \in [0,1].
$$ 
The bond overlap $Q_{1,1}$ will be defined in this case by
\begin{align*}
Q_{1,1}&=
\frac{1}{N^2} \sum_{1\leq i,j\leq N} (\sigma_i^1,\sigma_j^1) (\rho_i^1,\rho_j^1)
=
\sum_{u,v=1}^d\Bigl(\frac{1}{N}\sum_{i=1}^N\sigma_{i,u}^1\rho_{i,v}^1\Bigr)^2,
\end{align*}
and it is easy to see that Theorem \ref{thm1} implies that
\begin{align}
\e\Bigl\la \Bigl(Q_{1,1}-\bigl\la Q_{1,1}\bigr\ra \Bigr)^2 \Bigr\ra
&\leq \frac{4(\beta_1+\beta_2)}{\beta_1\beta_2\sqrt{N(1-t)}}
\end{align}
for $t<1$.

\end{example}
\begin{example}[Diluted $p$-spin model]\rm 
Let $\pi(\lambda N)$ be a Poisson random variable with mean $\lambda N$ and $(i_{j,k})_{j,k\geq 1}$ be i.i.d. uniform random variables on $\{1,\ldots,N\}.$ Consider two diluted $p$-spin models, 
\begin{align*}
H_1(\sigma)&=\beta_1\sum_{k\leq \pi(\lambda N)}g_k^1\sigma_{i_{1,k}}\cdots\sigma_{i_{p,k}},\,\,
G_1(\sigma)=\frac{\exp H_1(\sigma)}{Z_1},\\
H_2(\rho)&=\beta_2\sum_{k\leq \pi(\lambda N)}g_k^2\rho_{i_{1,k}}\cdots\rho_{i_{p,k}},\,\,
G_2(\rho)=\frac{\exp H_2(\rho)}{Z_2},
\end{align*}
where $\beta_1,\beta_2>0$ and $(g_k^1,g_k^2)_{k\geq 1}$ are independent Gaussian random vectors with covariance 
$$
\e(g_k^1)^2=\e (g_k^2)^2=1 \,\,\mbox{ and }\,\, \e g_k^1g_k^2=t \in [0,1].
$$ 
If we define the bond overlap $Q_{1,1}$ by 
\begin{align*}
Q_{1,1}&=\frac{1}{\pi(\lambda N)}\sum_{k=1}^{\pi(\lambda N)}\sigma_{i_{1,k}}^1\cdots\sigma_{i_{p,k}}^1\rho_{i_{1,k}}^1\cdots\rho_{i_{p,k}}^1
\end{align*}
when $\pi(\lambda N)\geq 1$, and $Q_{1,1} = 1$ (or any constant) when $\pi(\lambda N) = 0$, then applying Theorem \ref{thm1} conditionally on $\pi(\lambda N)$ and then averaging in $\pi(\lambda N)$ implies that for $t<1,$
\begin{align}
\e\Bigl\la \Bigl(Q_{1,1}-\bigl\la Q_{1,1}\bigr\ra \Bigr)^2\Bigr\ra 
&\leq \frac{4(\beta_1+\beta_2)}{\beta_1\beta_2\sqrt{1-t}}\e\frac{1}{\sqrt{\pi(\lambda N)}}{\rm I}(\pi(\lambda N)\geq 1).
\end{align}
The last expectation is of order $1/\sqrt{\lambda N}$ and, in fact, it is easy to check that it is bounded by $1/(\sqrt{\lambda N} - \sqrt{2/(\lambda N)}).$

\end{example}

\begin{example}[Edwards-Anderson model]\rm
Let $(V,E)$ be an arbitrary undirected finite graph and let $\beta_1,\beta_2,h_1,h_2\geq 0$. Consider two Edwards-Anderson models on $\{-1,+1\}^V$ with Gaussian random external fields,
\begin{align*}
H_1(\sigma)&=\beta_1\sum_{(i,j)\in E}g_{i,j}^1\sigma_i\sigma_j+h_1\sum_{i\in V}g_i^1\sigma_i,\,\,
G_1(\sigma)=\frac{\exp H_1(\sigma)}{Z_1}\\
H_2(\rho)&=\beta_2\sum_{(i,j)\in E}g_{i,j}^2\rho_i\rho_j+h_2\sum_{i\in V}g_i^2\rho_i,\,\,
G_2(\rho)=\frac{\exp H_2(\rho)}{Z_2},
\end{align*}
where $(g_{i,j}^1,g_{i,j}^2)$ are independent Gaussian random vectors with covariance 
$$
\e(g_{i,j}^1)^2=\e(g_{i,j}^2)^2=1 \,\,\mbox{ and }\,\, \e g_{i,j}^1g_{i,j}^2=t_E \in [0,1],
$$ 
$(g_i^1,g_i^2)$ are independent Gaussian random vectors with covariance 
$$
\e(g_i^1)^2=\e(g_i^2)^2=1 \,\,\mbox{ and }\,\, \e g_i^1g_i^2=t_V\in [0,1],
$$ 
and these two families of random vectors are independent of each other. From Theorem \ref{thm1}, we can deduce two kinds of quenched disorder chaos. First, if $\beta_1,\beta_2>0$ and $t_E<1,$ we obtain 
\begin{align}\label{ex2:eq1}
\e\Bigl\la \Bigl(Q_{1,1}-\bigl\la Q_{1,1}\bigr\ra \Bigr)^2\Bigr\ra & \leq \frac{4(\beta_1+\beta_2)}{\beta_1\beta_2\sqrt{|E|(1-t_E)}},
\end{align}
where $Q_{1,1}$ is the bond overlap
\begin{align*}
Q_{1,1}=\frac{1}{|E|}\sum_{(i,j)\in E}\sigma_i^1\sigma_j^1\rho_i^1\rho_j^1.
\end{align*}
If $h_1,h_2>0$ and $t_V<1$, then 
\begin{align}
\e\Bigl\la \Bigl(R_{1,1}-\bigl\la R_{1,1}\bigr\ra \Bigr)^2\Bigr\ra &\leq \frac{4(h_1+h_2)}{h_1h_2\sqrt{|V|(1-t_V)}},
\end{align}
where
$$
R_{1,1}=\frac{1}{|V|}\sum_{i\in V}\sigma_i^1\rho_i^1
$$
is the usual site overlap. The bound in (\ref{ex2:eq1}) was the one discussed in the introduction. 

\medskip
\noindent
\textbf{Remark.} In Theorem 1.6 of the same paper \cite{Chatt09}, Chatterjee also proved the following result. If $d$ is the maximum degree of the graph $(V,E)$ and
$$
q = \min\Bigl(\beta^2, \frac{1}{4d^2}\Bigr)
$$
then for some choice of absolute constant $C$,
\begin{equation}
\e \bigl\la Q_{1,1}\bigr\ra \geq Cq t^{1/(Cq)}.
\label{eqChatt2}
\end{equation}
If $d$ is fixed (for example, in the EA model on a finite dimensional lattice) and $t>0$ then (\ref{eqChatt2}) combined with (\ref{eqChatt1}) excludes the possibility that $Q_{1,1}$ concentrates near $0$ for large $|E|$, since the quenched average $\la Q_{1,1}\ra$ must be strictly positive with positive probability. This seems to be in contrast with the predictions of Fisher, Huse \cite{FH86} and Bray, Moore \cite{BM87} for the site overlap
\begin{equation}
R_{\ell,\ell'} = \frac{1}{|V|} \sum_{i\in V} \sigma_i^{\ell} \rho_i^{\ell'},
\end{equation}
which is expected to concentrate near zero when $t<1$. One interpretation of (\ref{eqChatt2}) is that there is no disorder chaos for the bond overlap. Another possible interpretation could be that the vectors $(\sigma_i^1\sigma_j^1)$ and $(\rho_i^1\rho_j^1)$ might have `preferred directions' and the overlap $\la Q_{1,1}\ra$ of their Gibbs averages $(\la \sigma_i^1\sigma_j^1 \ra)$ and $(\la \rho_i^1\rho_j^1 \ra)$ could deviate from zero but, otherwise, they have no common structure, which is some sort of weak disorder chaos. To strengthen this statement, one could also try to show that $\la Q_{1,1}\ra$ concentrates around its expected value $\e \la Q_{1,1}\ra$. 
\end{example}

\section{Self-averaging of the magnetization}\label{SecLabelMag}

From now on we will consider one system with the Hamiltonian as in (\ref{eq1}),
$$
Y(\sigma)=\sum_{e\in E}g_e f_e(\sigma),
$$
and the Gibbs measure as in (\ref{Geq1}),
$$
dG(\sigma)=\frac{\exp \gamma Y(\sigma)}{Z} d\mu(\sigma).
$$
Consider a vector $a = (a_e)_{e\in E}$ of some arbitrary constants and denote 
$$
\|a\|_2 = \Bigl(\sum_{e\in E} a_e^2\Bigr)^{1/2}
\,\,\mbox{ and }\,\,
\|a\|_1 = \sum_{e\in E} |a_e|.
$$
We will define a weighted bond magnetization by
\begin{equation}
m(\sigma) = \sum_{e\in E} a_e f_e(\sigma).
\end{equation}
The following holds.
\begin{theorem}\label{ThMag}
If $\gamma > 0$ then
\begin{equation}
\e\bigl\la \bigl(m(\sigma)-\bigl\la m(\sigma) \bigr\ra \bigr)^2\bigr\ra 
\leq  \frac{1}{\gamma} \|a\|_2 \|a\|_1.
\label{cor1:eq1}
\end{equation}
\end{theorem}
\begin{proof} If we consider the random variable $ g= \sum_{e} a_e g_e$ then Gaussian integration by parts gives
\begin{align*}
 \e \bigl \la m(\sigma^1) g \bigr\ra
& =
 \gamma\,  \e \bigl \la m(\sigma^1)^2 - m(\sigma^1) m(\sigma^2) \bigr\ra
\\
& =
\gamma\, \e\bigl\la \bigl(m(\sigma)-\bigl\la m(\sigma) \bigr\ra \bigr)^2\bigr\ra.
\end{align*}
On the other hand, since
\begin{align}
\label{eq:add}
|\la m(\sigma^1)\ra|&\leq \sum_{e\in E}|a_e||\la f_e(\sigma^1)\ra|\leq \|a\|_1,
\end{align}
we can write
$$
\bigl | \e \bigl \la m(\sigma^1) g \bigr\ra \bigr| =\bigl | \e \bigl \la m(\sigma^1) \bigr\ra g\bigr|
\leq \|a\|_1\e |g| \leq \|a\|_1 (\e g^2)^{1/2} = \|a\|_2 \|a\|_1
$$
and the proof follows.
\end{proof}

\begin{example}\rm
Consider the mixed $p$-spin model as in the Example 1 above. Let us consider $b_1,\ldots,b_N$ such that $\sum_{i=1}^N |b_i|=1.$  If we denote $\gamma=\beta_p/N^{(p-1)/2}$, let 
$$
\mbox{
$f_e(\sigma)=\sigma_{i_1}\cdots\sigma_{i_p}$ and $a_e=b_{i_1}\cdots b_{i_p}$ for $e=(i_1,\ldots,i_p)\in E=\{1,\ldots,N\}^p$ 
}
$$
then the bond magnetization is given by
\begin{align*}
m(\sigma)=
\sum_{1\leq i_1,\ldots,i_p\leq N} b_{i_1}\cdots b_{i_p} \sigma_{i_1}\cdots\sigma_{i_p}
=
\Bigl(\sum_{i=1}^Nb_i \sigma_i\Bigr)^p
\end{align*}
and \eqref{cor1:eq1} implies that, for $\beta_p>0$,
\begin{align}
\e\bigl\la \bigl(m(\sigma)-\bigl\la m(\sigma) \bigr\ra \bigr)^2\bigr\ra 
\leq 
\frac{\|a\|_2}{\gamma}=\frac{N^{(p-1)/2}\|b\|_2^p}{\beta_p}.
\end{align}
If we take $b_i = 1/N$, the bound becomes $(\beta_p \sqrt{N})^{-1}$ and $m(\sigma)$ is the $p$th power of the usual total site magnetization $N^{-1}\sum_{i\leq N} \sigma_i$. For odd $p$, this implies quenched self-averaging for the total site magnetization and, for even $p$, quenched self-averaging for its absolute value.
\end{example}

\section{Self-averaging of the site overlap assuming positive spin correlation}\label{SecLabelFKG}

Theorem \ref{ThMag} can be used to give a simplified proof of a slightly improved version of Lemma $2.6$ in \cite{RFIM}. Consider a finite set $V$ and consider any model with the Hamiltonian defined on $\sigma\in \{-1,+1\}^V$ that includes a Gaussian random field term,
\begin{equation}
H(\sigma) = H'(\sigma) + h \sum_{i\in V} g_i \sigma_i,
\label{rfHam}
\end{equation}
where $(g_i)$ are i.i.d. standard Gaussian random variables, independent of $H'(\sigma)$. For the next result, let us assume that the spins are positively correlated under the Gibbs measure,
\begin{equation}
\bigl\la\sigma _i\sigma_j \bigr\ra \geq \bigl\la\sigma_i \bigr\ra \bigl\la\sigma_j\bigr\ra \,\mbox{ for all $i,j\in V.$}
\label{FKG}
\end{equation}
For example, this was the case for the random field Ising model considered in \cite{RFIM} by the FKG inequality \cite{FKG}. Let
$$
R_{1,2} = \frac{1}{|V|}\sum_{i\in V} \sigma_i^1 \sigma_i^2
$$ 
denote the usual site overlap of two replicas. 
\begin{theorem} 
If the inequalities (\ref{FKG}) hold then
\begin{align}
\e \bigl\la \bigl(R_{1,2}-\la R_{1,2}\ra \bigr)^2 \bigr\ra\leq \frac{2}{h\sqrt{|V|}}.
\label{FKGconc}
\end{align}
\end{theorem}
\noindent
In particular, this removes the factor $\sqrt{2+h^2}$ from the bound in Lemma 2.6 in \cite{RFIM}.

\noindent
\begin{proof}
To prove this, we start by copying the following equation from the proof of Lemma 2.6 in \cite{RFIM}:
\begin{align*}
\e \bigl( \bigl\la R_{1,2}^2\bigr\ra -\bigl\la R_{1,2}\bigr\ra^2 \bigr)
&= 
\frac{1}{|V|^2}\sum_{i,j} \e \bigl( \la \sigma_i \sigma_j\ra^2 -\la\sigma_i\ra^2 \la\sigma_j\ra^2 \bigr)
\\
&= \frac{1}{|V|^2}\sum_{i,j} \e \bigl| \la \sigma_i \sigma_j -\la\sigma_i\ra \la\sigma_j\ra \bigr| 
\bigl| \la \sigma_i \sigma_j + \la\sigma_i\ra \la\sigma_j\ra \bigr|
\\
&\leq \frac{2}{|V|^2}\sum_{i,j}\e \bigl| \la\sigma_i\sigma_j\ra-\la\sigma_i\ra\la\sigma_i\ra \bigr|
\\
&=\frac{2}{|V|^2}\sum_{i,j}\e \bigl( \la\sigma_i\sigma_j\ra-\la\sigma_i\ra\la\sigma_j\ra \bigr),
\end{align*}
where in the last step the positive correlation condition (\ref{FKG}) was used. Next, if we consider the magnetization $m(\sigma)=|V|^{-1}\sum_i\sigma_i$ then
\begin{align*}
\e\bigl\la \bigl(m(\sigma)-\bigl\la m(\sigma) \bigr\ra \bigr)^2\bigr\ra 
&=\frac{1}{|V|^2}\sum_{i,j}\e \bigl\la \bigl(\sigma_i-\la \sigma_i\ra \bigr) \bigl(\sigma_j-\la \sigma_j\ra \bigr)\bigr\ra
\\
&=\frac{1}{|V|^2}\sum_{i,j}\e \bigl(\la\sigma_i\sigma_j\ra-\la\sigma_i\ra\la\sigma_j\ra \bigr).
\end{align*}
Therefore, the inequalities (\ref{FKG}) imply that
\begin{align*}
\e \bigl( \bigl\la R_{1,2}^2\bigr\ra -\bigl\la R_{1,2}\bigr\ra^2 \bigr)
\leq 
2 \e\bigl\la \bigl(m(\sigma)-\bigl\la m(\sigma) \bigr\ra \bigr)^2\bigr\ra.
\end{align*}
Finally, using Theorem \ref{ThMag} with $\gamma = h$ and $Y(\sigma) = \sum_{i\in V} g_i \sigma_i$ implies (\ref{FKGconc}). 
\end{proof}

\section{Self-averaging of random fields}\label{SecLabHam}

Throughout this section, we will use the integration by parts formula 
\begin{equation}
\e H_k(g) F(g) = \e H_{k-1}(g) F'(g)
\label{Hermite}
\end{equation}
for the Hermite polynomials 
$$
H_k(x) = (-1)^k e^{x^2/2} \frac{d^k}{dx^k} e^{-x^2/2}
$$ 
of degree $k\geq 1$. In (\ref{Hermite}), $g$ is a standard Gaussian random variable and $F$ is a continuously differentiable function such that $F'$ is of moderate growth. The case $k=1$,
\begin{equation}
\e gF(g) = \e F'(g),
\end{equation} 
is often called the (usual) Gaussian integration by parts, and
\begin{equation}
\e (g^2-1) F(g) = \e g F'(g)
\label{Hermitek2}
\end{equation}
corresponds to the case $k=2$. 

\smallskip
Let $Y(\sigma)$ and $dG(\sigma)$ be as in Section \ref{SecLabelMag}. Consider a random field
\begin{equation}
W(\sigma) = \sum_{e\in E} a_e g_e f_e(\sigma)
\label{Wsigma}
\end{equation}
for arbitrary constants $a_e$ for $e\in E$. Denote 
$$
\|a\|_2 = \Bigl(\sum_{e\in E} a_e^2\Bigr)^{1/2}
\,\,\mbox{ and }\,\,
\|a\|_1 = \sum_{e\in E} |a_e|.
$$
We will start with the following.
\begin{theorem}\label{ThHSA0}
If $\gamma>0$ then
\begin{equation}
\e \bigl( \bigl\la W(\sigma)^2\bigr\ra - \bigl\la W(\sigma)\bigr\ra^2 \bigr)
\leq
\frac{\sqrt{2}}{\gamma} \|a\|_2\|a\|_1.
\label{ThHSAeq0}
\end{equation}
\end{theorem}
\begin{proof} 
Using the integration by parts formula in (\ref{Hermitek2}), we can write
$$
\e (g_e^2-1)\bigl\la W(\sigma)\bigr\ra =  \e g_e \bigl\la a_e f_e(\sigma)\bigr\ra
+ \gamma\, \e g_e \bigl\la W(\sigma) f_e(\sigma)\bigr\ra
- \gamma\, \e g_e \bigl\la W(\sigma^1) f_e(\sigma^2)\bigr\ra.
$$
Multiplying both sides by $a_e$ and summing over $e\in E$ gives
$$
\e \sum_{e\in E} a_e (g_e^2-1)\bigl\la W(\sigma)\bigr\ra = 
\sum_{e\in E} a_e^2 \e \bigl\la g_e f_e(\sigma)\bigr\ra 
+ \gamma\, \e \bigl\la W(\sigma)^2\bigr\ra - \gamma\, \e \bigl\la W(\sigma)\bigr\ra^2
$$
and, therefore,
\begin{align}\label{eq0a}
\gamma\, \e \bigl( \bigl\la W(\sigma)^2\bigr\ra - \bigl\la W(\sigma)\bigr\ra^2 \bigr)
=
\e \sum_{e\in E} a_e (g_e^2-1)\bigl\la W(\sigma)\bigr\ra - \sum_{e\in E} a_e^2 \e \bigl\la g_e f_e(\sigma)\bigr\ra.
\end{align}
By the usual Gaussian integration by parts,
$$
\e \bigl\la g_e f_e(\sigma)\bigr\ra = \gamma\, \e \bigl( \bigl\la f_e(\sigma)^2\bigr\ra -\bigl\la f_e(\sigma)\bigr\ra^2\bigr)
\geq 0,
$$
so omitting the last sum in (\ref{eq0a}) yields an upper bound
$$
\gamma\, \e \bigl( \bigl\la W(\sigma)^2\bigr\ra - \bigl\la W(\sigma)\bigr\ra^2 \bigr)
\leq
\e \sum_{e\in E} a_e (g_e^2-1) \bigl\la W(\sigma)\bigr\ra.
$$
Let us note that
\begin{align*}
\e \Bigl( \sum_{e\in E} a_e (g_e^2-1)\Bigr)^2&=\sum_{e,e'\in E}a_e a_{e'}\e(g_e^2-1)(g_{e'}^2-1)=2\|a\|_2^2, 
\end{align*}
since the terms for $e\not = e'$ are equal to $0$ and $\e (g_e^2-1)^2 = 2$. By the Cauchy-Schwarz inequality,
$$
\gamma\, \e \bigl( \bigl\la W(\sigma)^2\bigr\ra - \bigl\la W(\sigma)\bigr\ra^2 \bigr)
\leq
\sqrt{2}\|a\|_2 \bigl(\e \bigl\la W(\sigma) \bigr\ra^2\bigr)^{1/2}.
$$
Finally, using that $|W(\sigma)| \leq \sum_{e\in E} |a_e g_e|$ and $\e ( \sum_{e\in E} |a_e g_e|)^2 \leq \|a\|_1^2$ finishes the proof.
\end{proof}

\begin{example}\rm
If in (\ref{ThHSAeq}) we take all $a_e=1$, we get
\begin{equation}
\e \bigl( \bigl\la Y(\sigma)^2\bigr\ra - \bigl\la Y(\sigma)\bigr\ra^2 \bigr)
\leq
\frac{\sqrt{2}|E|^{3/2}}{\gamma}.
\label{ThHSAeq2}
\end{equation}
Applying this to the Hamiltonian (\ref{rfHam}) with the Gaussian random field gives a new proof of Lemma 2.9 in \cite{RFIM}. If in Theorem \ref{ThHSA} we take $E=V$, $\gamma = h$, for $i\in V$ take $f_i(\sigma) = \sigma_i$, and divide both sides of (\ref{ThHSAeq}) by $|V|^2$, then the normalized random field
$$
h(\sigma) = \frac{1}{|V|} \sum_{i\in V} g_i \sigma_i
$$
satisfies
\begin{equation}
\e \bigl( \bigl\la h(\sigma)^2\bigr\ra - \bigl\la h(\sigma)\bigr\ra^2 \bigr)
\leq
\frac{\sqrt{2}}{h\sqrt{|V|}}.
\label{eqlast}
\end{equation}
This inequality was used in \cite{RFIM} to establish `half' of the Ghirlanda-Guerra identities for the first moment of the overlaps, with the other half following from the existence of the limit for the free energy.
\end{example}

Next, we will show how one can push the above proof even further to improve the bound for small values of $\gamma$.
\begin{theorem}\label{ThHSA}
If $\gamma>0$ then
\begin{equation}
\e \bigl( \bigl\la W(\sigma)^2\bigr\ra - \bigl\la W(\sigma)\bigr\ra^2 \bigr)
\leq
\|a\|_2^2 + \sqrt{2}\|a\|_2\|a\|_1.
\label{ThHSAeq}
\end{equation}
\end{theorem}
\noindent
\textbf{Remark.} When all $a_e=1$, the bound becomes $|E| +\sqrt{2}|E|^{3/2}$, which is an improvement over (\ref{ThHSAeq2}) for small values of $\gamma$. In fact, for very small values of $\gamma$, if one simply integrates $\e \la W(\sigma)^2\ra $ by parts to obtain a trivial bound $\|a\|_2^2 + C \gamma^2 \|a\|_1^2$, this gives further improvement for very small values of $\gamma$.

\smallskip
\noindent
\begin{proof} To prove this inequality, let us look at the right hand side of (\ref{eq0a}) more closely. First,
\begin{align*}
\e \sum_{e\in E}a_e(g_e^2-1) \bigl\la W(\sigma) \bigr\ra=\e\sum_{e,e'\in E}a_ea_{e'}(g_e^2-1)g_{e'}
\bigl\la f_{e'}(\sigma) \bigr\ra.
\end{align*}
It will be convenient to introduce the notation
\begin{equation}
F_{e}= \frac{1}{\gamma}\frac{\partial}{\partial g_e} \bigl\la f_{e}(\sigma)\bigr\ra
= \bigl\la f_{e}(\sigma)^2 \bigr\ra- \bigl\la f_{e}(\sigma) \bigr\ra^2.
\label{Cee1}
\end{equation}
For the terms $e=e'$, using the formula (\ref{Hermitek2}) for the factors $g_e^2-1$ gives
\begin{align}
\sum_{e\in E}a_e^2 \e (g_e^2-1)g_e \bigl\la f_e(\sigma) \bigr\ra
&=\sum_{e\in E}a_e^2 \e  g_e \bigl\la f_e(\sigma) \bigr\ra
+\gamma\sum_{e\in E}a_e^2 \,\e g_e^2 F_e.
\label{Sum01}
\end{align}
Since $0\leq F_e\leq 1$, the second sum is bounded by $\gamma \|a\|_2^2$. The first sum cancels out the last sum in (\ref{eq0a}), so
\begin{align}
\gamma\, \e \bigl( \bigl\la W(\sigma)^2\bigr\ra - \bigl\la W(\sigma)\bigr\ra^2 \bigr)
\leq
\e\sum_{e\not =e'}a_ea_{e'}(g_e^2-1)g_{e'}\bigl\la f_{e'}(\sigma) \bigr\ra
+
\gamma \|a\|_2^2.
\label{SumSumineq}
\end{align}
If in the first term on the right hand side we use the usual Gaussian integration by parts with respect to $g_e'$, it can be rewritten as
\begin{align}
\sum_{e'\in E}a_{e'} \,\e \Bigl(\sum_{e\not =e'}a_e(g_e^2-1) \Bigr)g_{e'}\bigl\la f_{e'}(\sigma) \bigr\ra
=
\gamma \sum_{e'\in E}a_{e'} \,\e \Bigl(\sum_{e\not =e'}a_e(g_e^2-1) \Bigr) F_{e'}.
\label{Sum02}
\end{align}
Since $0\leq F_{e'}\leq 1$ and
$$
\e \Bigl(\sum_{e\not =e'}a_e(g_e^2-1) \Bigr)^2 \leq 2\|a\|_2^2,
$$
by the Cauchy-Schwarz inequality, we can bound the last sum by $\sqrt{2} \gamma \|a\|_2 \|a\|_1$ and this finishes the proof.
\end{proof} 

\begin{example}\rm
Using the bound (\ref{ThHSAeq}), one can supplement (\ref{eqlast}) in the Example 6 with
\begin{equation}
\e \bigl( \bigl\la h(\sigma)^2\bigr\ra - \bigl\la h(\sigma)\bigr\ra^2 \bigr)
\leq
\frac{1}{|V|} + \frac{\sqrt{2}}{\sqrt{|V|}}
\label{eqlast2}
\end{equation}
for small values of $h$. 
\end{example}

\smallskip
There is a natural generalization of the previous results to the random field
\begin{equation}
W(\sigma) = \sum_{e\in E}a_e H_k(g_e) f_e(\sigma),
\end{equation}
where, as above, $H_k$ is the Hermite polynomial of degree $k\geq 0.$ Let us denote
\begin{equation}
F_e^{(k)} = \frac{1}{\gamma^k}\frac{\partial^k}{\partial g_e^k} \bigl\la f_e(\sigma)\bigr\ra,
\end{equation}
and let $C_k$ be a constant such that $|F_e^{(k)}| \leq C_k$ with probability one. For example, $F_e^{(0)}=\la f_e(\sigma)\ra$ and $C_0=1$ (this was used in \eqref{eq:add}) and $F_e^{(1)}=\la f_{e}(\sigma)^2 \ra- \la f_{e}(\sigma) \ra^2$ with $C_1=1$, which already appeared in (\ref{Cee1}). The following analogue of Theorems \ref{ThMag}, \ref{ThHSA0} and \ref{ThHSA} holds in this case.
\begin{theorem}\label{ThLast}
We have that for $k\geq 0,$
\begin{align}\label{ineq1}
\e \bigl( \bigl\la W(\sigma)^2\bigr\ra - \bigl\la W(\sigma)\bigr\ra^2 \bigr)
&\leq 
\frac{\sqrt{k! (k+1)!}}{\gamma}  \|a\|_1\|a\|_2
\end{align}
and for $k\geq 1,$
\begin{align}
\label{ineq2}
\e \bigl( \bigl\la W(\sigma)^2\bigr\ra - \bigl\la W(\sigma)\bigr\ra^2 \bigr)
&\leq C_k \sqrt{(k+1)!} \gamma^{k-1} \|a\|_1\|a\|_2 + k! \|a\|_2^2.
\end{align}
\end{theorem}
\begin{proof}
Using the integration by parts formula (\ref{Hermite}),
\begin{align*}
&\e H_{k+1}(g_e)\la W(\sigma)\ra = \e H_k(g_e)\frac{\partial}{\partial g_e} \bigl\la W(\sigma) \bigr\ra\\
&=
a_e\e H_k(g_e)H_{k}'(g_e) \bigl\la f_e(\sigma) \big\ra
+\gamma\, \e H_k(g_e) \bigl( \bigl\la W(\sigma)f_e(\sigma) \bigr\ra- \bigl\la W(\sigma^1)f_e(\sigma^2)\bigr\ra \bigr).
\end{align*}
Multiplying both sides by $a_e$ and summing over $e\in E$ gives
\begin{align*}
\e \sum_{e\in E}a_e H_{k+1}(g_e) \bigl\la W(\sigma) \bigr\ra
&=\e \sum_{e\in E}a_e^2H_k(g_e)H_{k}'(g_e) \bigl\la f_e(\sigma) \bigr\ra
+\gamma\, \e \bigl( \bigl\la W(\sigma)^2\bigr\ra - \bigl\la W(\sigma)\bigr\ra^2 \bigr)
\end{align*}
and, therefore,
\begin{align}
\gamma\, \e \bigl( \bigl\la W(\sigma)^2\bigr\ra - \bigl\la W(\sigma)\bigr\ra^2 \bigr)
= &\,\,
\e \sum_{e\in E}a_eH_{k+1}(g_e) \bigl\la W(\sigma)\bigr\ra
\nonumber
\\
&-\, \e \sum_{e\in E}a_e^2H_k(g_e)H_{k}'(g_e) \bigl\la f_e(\sigma) \bigr\ra.
\label{eq1ag}
\end{align}
As above, the first term can be bounded as follows,
\begin{align*}
\e \sum_{e\in E}a_eH_{k+1}(g_e) \bigl\la W(\sigma)\bigr\ra
&\leq 
\Bigl(\e\Bigl(\sum_{e\in E}a_eH_{k+1}(g_e)\Bigr)^2\Bigr)^{1/2}\Bigl(\e \bigl\la W(\sigma) \bigr\ra^2\Bigr)^{1/2}
\\
&\leq 
\Bigl(\sum_{e\in E}a_e^2 \e H_{k+1}(g_e)^2\Bigr)^{1/2} \Bigl(\e \Bigl(\sum_{e\in E}|a_e||H_k(g_e)|\Bigr)^2\Bigr)^{1/2}
\\
&\leq \sqrt{(k+1)!} \|a\|_2 \sqrt{k!}\|a\|_1,
\end{align*}
where we used that $\e H_\ell(g)^2 = \ell!$ for $\ell=k, k+1$. This will finish the proof of \eqref{ineq1} if we can show that the second term is negative. Using (\ref{Hermite}) for the factor $H_k(g_e)$ gives
$$
\e H_k(g_e) H_{k}'(g_e) \bigl\la f_e(\sigma)\bigr\ra
=
\e H_{k-1}(g_e) H_{k}''(g_e) \bigl\la f_e(\sigma)\bigr\ra
+
\gamma\, \e H_{k-1}(g_e) H_{k}'(g_e) F_e^{(1)}.
$$
Using a well-known relationship $H_k(x)' = k H_{k-1}(x)$, we can rewrite this as
$$
\e H_k(g_e) H_{k}'(g_e) \bigl\la f_e(\sigma)\bigr\ra
=
k \e H_{k-1}(g_e) H_{k-1}'(g_e) \bigl\la f_e(\sigma)\bigr\ra
+
k \gamma\, \e H_{k-1}(g_e)^2 F_e^{(1)}.
$$
Finally, using that $F_e^{(1)} \geq 0$ and proceeding by induction on $k$, we get
$$
\e H_k(g_e) H_{k}'(g_e) \bigl\la f_e(\sigma)\bigr\ra
\geq 
k \e H_{k-1}(g_e) H_{k-1}'(g_e) \bigl\la f_e(\sigma)\bigr\ra
\geq 0.
$$

\smallskip
To obtain \eqref{ineq2}, we need further calculations for the first term on the right-hand side of \eqref{eq1ag}. Let us begin by writing
\begin{align}
\e \sum_{e\in E}a_eH_{k+1}(g_e) \bigl\la W(\sigma) \bigr\ra
= &\,\,
\e\sum_{e\neq e'}a_e a_{e'}H_{k+1}(g_e)H_k(g_{e'}) \bigl\la f_{e'}(\sigma)\bigr\ra
\nonumber
\\
&+\,
\e\sum_{e\in E} a_e^2 H_{k+1}(g_e)H_k(g_{e}) \bigl\la f_{e}(\sigma)\bigr\ra.
\label{drepr}
\end{align}
Using the integration by parts formula (\ref{Hermite}) repeatedly, for any $e\neq e'$ we get
\begin{align*}
\e H_{k+1}(g_e)H_k(g_{e'}) \bigl\la f_{e'}(\sigma)\bigr\ra
&=
\e H_{k+1}(g_e)\frac{\partial^k}{\partial g_{e'}^{k}} \bigl\la f_{e'}(\sigma)\bigr\ra
=
\gamma^k \e H_{k+1}(g_e) F_{e'}^{(k)}.
\end{align*}
Using (\ref{Hermite}) once for the factor $H_{k+1}(g_e)$, for $e=e'$ we get
\begin{align*}
\e H_{k+1}(g_e)H_k(g_{e}) \bigl\la f_{e}(\sigma)\bigr\ra
&=
\e H_k(g_e)H_k'(g_e)\bigl\la f_e(\sigma) \bigr\ra
+ \gamma\, \e H_{k}(g_e)^2F_e^{(1)}.
\end{align*}
Therefore, we can rewrite (\ref{drepr}) as
\begin{align*}
\e \sum_{e\in E}a_eH_{k+1}(g_e) \bigl\la W(\sigma) \bigr\ra
& = 
\gamma^k \e \sum_{e\not =e'}a_e a_{e'}H_{k+1}(g_e) F_{e'}^{(k)}
\\
& \quad + 
\e \sum_{e\in E}a_e^2H_k(g_e)H_k'(g_e) \bigl\la f_e(\sigma)\bigr\ra
+\gamma\, \e \sum_{e\in E}a_e^2 H_{k}(g_e)^2F_e^{(1)}.
\end{align*}
Plugging this into \eqref{eq1ag} and dividing both sides by $\gamma$,
\begin{align*}
\e \bigl( \bigl\la W(\sigma)^2\bigr\ra - \bigl\la W(\sigma)\bigr\ra^2 \bigr)
& = 
\gamma^{k-1} \e \sum_{e\not =e'}a_e a_{e'}H_{k+1}(g_e) F_{e'}^{(k)}
+
\e \sum_{e\in E}a_e^2 H_{k}(g_e)^2F_e^{(1)}.
\end{align*}
Since $F_e^{(1)}\leq 1$, the second term can be bounded by 
\begin{align*}
\e \sum_{e\in E}a_e^2 H_{k}(g_e)^2F_e^{(1)} \leq k! \|a\|_2^2.
\end{align*}
To bound the first term, let us rewrite it as
$$
\gamma^{k-1} \e \sum_{e\not =e'}a_e a_{e'}H_{k+1}(g_e) F_{e'}^{(k)}
=
\gamma^{k-1} \sum_{e'\in E} a_{e'} \e \, F_{e'}^{(k)} \sum_{e\not = e'}a_e H_{k+1}(g_e).
$$
Since, for any fixed $e'\in E$, $|F_{e'}^{(k)}|\leq C_k$ and 
$$
\e \Bigl(\sum_{e\not = e'}a_eH_{k+1}(g_e)\Bigr)^2
\leq 
(k+1)!\|a\|_2^2,
$$
this can be bounded by $C_k \sqrt{(k+1)!} \gamma^{k-1} \|a\|_2 \|a\|_1$, which finishes the proof of \eqref{ineq2}.
\end{proof}

\end{document}